\documentclass[12pt]{amsart} 
\usepackage{amsmath,amssymb,amscd,amsthm}
\usepackage{latexsym}
\usepackage{graphicx}
\usepackage[english]{babel}
\usepackage[latin1]{inputenc}       
\usepackage{colortbl}

\usepackage[a4paper,twoside,left=3cm,right=2.8cm,top=3.1cm,bottom=2.3cm]{geometry}

\newtheorem{theorem}{Theorem}[section]
\newtheorem{corollary}[theorem]{Corollary}
\newtheorem{proposition}[theorem]{Proposition}
\newtheorem{lemma}[theorem]{Lemma}

\newtheorem{remark}[theorem]{Remark}

\def\irr#1{{\rm Irr}(#1)}
\def\irrr#1#2 {\irr {#1 \mid #2}}

\newcommand{\R}{\mathbb R}

    \def\Xint#1{\mathchoice
       {\XXint\displaystyle\textstyle{#1}}%
       {\XXint\textstyle\scriptstyle{#1}}%
       {\XXint\scriptstyle\scriptscriptstyle{#1}}%
       {\XXint\scriptscriptstyle\scriptscriptstyle{#1}}%
       \!\int}
    \def\XXint#1#2#3{{\setbox0=\hbox{$#1{#2#3}{\int}$}
         \vcenter{\hbox{$#2#3$}}\kern-.5\wd0}}
    
    \def\dashint{\Xint-}
    
    \makeatother

\begin{document}

\global\long\def\RR{\mathbb{R}}%

\global\long\def\SS{\mathbb{S}}%

\global\long\def\oo{\mathbf{1}}%

\global\long\def\dx{\mathrm{d}x}%

\global\long\def\dd{\mathrm{d}}%

\global\long\def\var{\operatorname{Var}}%

\global\long\def\cov{\operatorname{Cov}}%

\global\long\def\div{\operatorname{div}}%

\global\long\def\tr{\operatorname{tr}}%

\global\long\def\II{\operatorname{II}}%

\global\long\def\proj{\operatorname{Proj}}%

\global\long\def\nn#1{\left\Vert #1\right\Vert }%

\global\long\def\PP{\mathbb{P}}%

\global\long\def\EE{\mathbb{E}}%

\title[Stability in the B-theorem]{Stability and the equality case in the B-theorem}
\author[Orli Herscovici, Galyna V. Livshyts, Liran Rotem, Alexander Volberg]{Orli Herscovici, Galyna V. Livshyts, Liran Rotem, Alexander Volberg}

\address{School of Mathematics, Georgia Institute of Technology, Atlanta, GA} \email{oherscovici3@gatech.edu}

\address{School of Mathematics, Georgia Institute of Technology, Atlanta, GA} \email{glivshyts6@math.gatech.edu}

\address{Faculty of Mathematics, Technion -- Israel Institute of Technology, Haifa, Israel} 
\email{lrotem@technion.ac.il}

\address{School of Mathematics, Michigan State University, East Lansing, MI} 
\email{volberg@msu.edu}

\begin{abstract} In this paper, we show the stability, and characterize the equality cases in the strong B-inequality of Cordero-Erasquin, Fradelizi and Maurey \cite{B-conj}. As an application, we establish uniqueness of Bobkov's maximal Gaussian measure position from \cite{Bobkov-Mpos}.
\end{abstract}
\maketitle

\section{Introduction}

Let $\gamma$ denote the standard Gaussian measure on $\R^n$ with density $(2\pi)^{-n/2} e^{-\frac{|x|^2}{2}}$. We say that a set $K$ in $\R^n$ is \emph{symmetric} if $K=-K$, i.e. for all $x \in K$ we have $-x \in K$. 

The B-theorem of Cordero-Erausquin, Fradelizi and Maurey \cite{B-conj} states that for every symmetric convex set $K\subset\R^n$, and 
every $a,b>0$,
\begin{equation}
\gamma\left(\sqrt{ab}K\right)\geq\sqrt{\gamma(aK)\gamma(bK)}. \label{B-statement}
\end{equation}
In other words, $\gamma(e^t K)$ is log-concave in $t$ when $K$ is a convex symmetric set.

The inequality (\ref{B-statement}) was first conjectured by Lata{\l}a in \cite{Latala2002}, who attributed the question to Banaszczyk. 
It is a strengthening of the inequality 
$$ \gamma\left(\frac{a+b}{2}K\right)\geq\sqrt{\gamma(aK)\gamma(bK)}, $$
which holds with no symmetry assumption and follows from the Pr\'{e}kopa--Leindler inequality \cite{Prekopa1971}, \cite{Leindler1972} -- see also Borell \cite{Borell1974, Borell1975} and Brascamp and Lieb \cite{BrLi}.

More generally, for a vector $x\in\R^n,$ we will use the notation
\begin{equation}\label{e^xK}
e^x K=\{(e^{x_1}y_1,...,e^{x_n}y_n):\,\, (y_1,...,y_n)\in K\}.
\end{equation}
The strong version of the B-theorem then states that for any pair $x,y\in\R^n,$
\begin{equation}\label{B-strong}
\gamma\left(e^{\frac{x+y}{2}}K\right)\geq\sqrt{\gamma(e^x K)\gamma(e^y K)}.
\end{equation}

Nayar and Tkocz \cite{Nayar2012} showed that the assumption that $K$ is symmetric is necessary and (\ref{B-statement}) does not necessarily 
hold if one replaces it with the weaker assumption $0 \in K$. A natural question, raised in \cite{B-conj}, is what other measures satisfy 
an inequality such as (\ref{B-statement}) or (\ref{B-strong}). In dimension $n=2$ it follows from the works of B\"or\"oczky, Lutwak, Yang
and Zhang \cite{Boroczky2012} and of Saroglou \cite{Saroglou2016} that all even log-concave measures satisfy (\ref{B-statement}). 
For uniform measures on convex bodies this was verified independently by Livne Bar-On \cite{LivneBar-On2013}. In arbitrary dimensions, 
Eskenazis, Nayar and Tkocz \cite{Eskenazis2018} proved that (\ref{B-strong}) holds for certain Gaussian mixtures, and in 
\cite{Cordero-Erausquin2021} it was proved in particular that all rotation-invariant log-concave measures satisfy (\ref{B-strong}). 

In this paper we will concentrate solely on the Gaussian case, and prove a stability version of this celebrated inequality.  Recall that the in-radius $r(K)$ of a 
symmetric convex body $K$ is the largest number $r>0$ such that $rB^n_2\subset K$, where $B^n_2$ denotes the unit Euclidean ball 
centered at the origin. Our theorem then reads:
\begin{theorem}\label{main-thm-global}
Suppose $0\leq a<b<\infty$ and let $K$ be a symmetric convex body. Suppose that
$$\gamma(\sqrt{ab}K)\leq\sqrt{\gamma(aK)\gamma(bK)}(1+\epsilon)$$
for small enough $\epsilon>0$. Then either the in-radius $r(K)$ satisfies
$$r(K)\geq \frac{1}{b} \sqrt{\log \left(\frac{c \log(b/a)^2}{n^2 \epsilon}\right)},$$ 
or 
$$r(K)\leq \frac{C\sqrt{n}}{a} \epsilon^{\frac{1}{n+1}}\left(\log({b}/{a})\right)^{-\frac{2}{n+1}}.$$
\end{theorem}

\begin{remark} We note that our estimate for $r$ is essentially sharp: indeed, consider first a case when $K$ is a symmetric strip $S_R$ of width $2R$, that is
$$S_R=\{x=(x_1,\ldots,x_n)\in\RR^n\,:\,x_1\in [-R,R]\}.$$
Note that 
\[
\gamma(S_R)=\frac{1}{\sqrt{2\pi}}\int_{-R}^Re^{-\frac{s^2}{2}}ds=1-\frac{2}{\sqrt{2\pi}}\int_R^\infty e^{-\frac{s^2}{2}}ds.
\]
Suppose $0< a\leq b\leq\infty$ and there exists $\epsilon>0$ such that $$\gamma(\sqrt{ab}S_R)=\sqrt{\gamma(aS_R)\gamma(bS_R)}(1+\epsilon).$$ 
Relaxing the upper and lower bound obtained by Szarek and Werner \cite{Szarek1999} for Komatsu inequality, we have that
\[
\frac{1}{R+1}\,e^{-\frac{R^2}{2}}\leq\int_{R}^{\infty}e^{-\frac{s^2}{2}}ds\leq\frac{1}{R}\,e^{-\frac{R^2}{2}}.
\]
Let $F(t)=\frac{2}{\sqrt{2\pi}}\int_{tR}^\infty e^{-\frac{s^2}{2}}\,ds$ for  $t>0$. Then
\[
\epsilon=\frac{\gamma(S_{\sqrt{ab}R})}{\sqrt{\gamma(S_{aR})\gamma(S_{bR})}}-1=\frac{1-F(\sqrt{ab})}{\sqrt{(1-F(a))(1-F(b))}}-1,
\]
and, since $F(a)>F(\sqrt{ab})>F(b)$, we get
\[
\epsilon\leq\frac{1-F(\sqrt{ab})}{1-F(a)}-1\leq\frac{F(a)}{1-F(a)},
\]
which leads to 
\[
\frac{\epsilon}{1+\epsilon}\leq
F(a)=\frac{2}{\sqrt{2\pi}}\int_{aR}^\infty e^{-\frac{s^2}{2}}ds\leq\frac{2}{\sqrt{2\pi}aR}\,e^{-\frac{a^2R^2}{2}}.
\] 
Thus, 

\[
\frac{\epsilon}{1+\epsilon}\leq
\frac{2}{\sqrt{2\pi}} \frac{1}{aR}e^{-\frac{a^2R^2}{2}},
\]
or, 
\[
R\leq C\,\sqrt{\log\left(1+\frac{1}{\epsilon}\right)}.
\]

However, the upper estimate on $r$ in Theorem \ref{main-thm-global} is likely not sharp: indeed, by considering $K=rB^n_2$ to be a very small ball, we only get that $r(K)\geq C\sqrt{\epsilon}$, where $\epsilon$ is the corresponding deficit in the B-inequality. 

\end{remark}

Theorem \ref{main-thm-global} implies:
\begin{corollary}\label{key-cor}
Suppose $0\leq a<b<\infty$ and $K$ is a symmetric convex set such that $\gamma(\sqrt{ab}K)=\sqrt{\gamma(aK)\gamma(bK)}$. 
Then either $K=\R^n$ or $K$ has an empty interior.
\end{corollary}
\begin{proof}
We can apply Theorem \ref{main-thm-global} with an arbitrary $\epsilon > 0$. Letting $\epsilon \to 0$ we see that either 
$r(K)=\infty$ so $K=\R^n$, or $r(K)=0$ so $K$ has an empty interior. 
\end{proof}

We also prove a stability version for the strong B-theorem. Given a convex body $K$ and a vector $x\in\partial K$, 
the outer unit normal to $\partial K$ (i.e. the Gauss map) at $x$ will be denoted by $n_x=(n^1_x, ..., n^n_x)$. 
It is well known (see e.g. \cite{Schneider2013}) that $n_x$ is uniquely defined almost everywhere on $\partial K$. Recall that for an $(n-1)$-dimensional surface $M,$ $\gamma^+(M)$ denotes the Gaussian perimeter of $M.$ Our theorem reads: 

\begin{theorem}\label{main-thm-strong-global}
Fix parameters $\delta, \alpha, \beta > 0$ Fix $x,y\in\R^n$ such that $|e^x|\leq |e^y|$ (where we use the notation
$e^x=(e^{x_1},...,e^{x_n})$, as well as (\ref{e^xK})). Let $K$ be a symmetric convex body, and suppose that
$$\gamma(e^{\frac{x+y}{2}}K)\leq\sqrt{\gamma(e^x K)\gamma(e^y K)}(1+\epsilon)$$ 
for small enough $\epsilon>0$. Consider
$$\sigma^{\delta}=\{i\in [n]:\, |x_i-y_i|\geq \delta\},$$
and let
$$\Omega_{\delta,\alpha}(K)=\bigl\{x\in \partial K:\,\sum_{i\in\sigma^{\delta}} (n^i_x)^2\geq \alpha\bigr\}.$$
Then either there exists a vector $z\in [x,y]$ such that
$$\gamma^+(\Omega_{\delta, \alpha}(e^z K))\leq \beta\gamma^+(\partial (e^z K)),$$
or
$$r(K)\geq |e^y|^{-1}\sqrt{\log \frac{\delta^2\alpha \beta }{\epsilon n^2}},$$ 
or 
$$r(K)\leq C |e^x|^{-1} \sqrt{n} \epsilon^{\frac{1}{n+1}}\left(\delta^2\alpha\beta \right)^{-\frac{1}{n+1}}.$$
\end{theorem}

As a corollary, we shall deduce:

\begin{corollary}\label{strong-equal}
Let $x, y\in\R^n$, set $\sigma_{x,y}=\{j\in [n]:\, x_j\neq y_j\}$, and let $K$ be a symmetric convex set in $\R^n.$ Then
$$\gamma\left(e^{\frac{x+y}{2}}K\right)=\sqrt{\gamma\left(e^x K\right)\gamma\left(e^y K\right)},$$
if and only if either $K$ has an empty interior, or $K=\R^n,$ or (more generally) $K=K_0\times H_{x,y},$ with $K_0\subset H_{x,y}^{\perp},$ where $H_{x,y}=\{z\in\R^n: z_j=0\,\, \forall j\in\sigma_{x,y}\}.$
\end{corollary}

We apply Corollary \ref{strong-equal} to show uniqueness of the Bobkov maximal Gaussian measure position (MGM from now on) for a convex body $K.$ Recall from \cite{Bobkov-Mpos} that a symmetric convex body $K$ is said to be in Bobkov's MGM position if for any volume preserving linear operator $T$ on $\R^n,$ we have $\gamma(K)\geq \gamma(TK).$ Bobkov showed that $K$ is in the MGM position if and only if the restriction of the Gaussian measure on $K$ is isotropic (recall that a measure is isotropic if its barycenter is at the origin, and the covariance matrix is proportional to the identity). Isotropicity often arises when the measure is placed in some optimizing position, see e.g. \cite{AGM-book}, \cite{Shiri-Katzin}, \cite{Shiri-Eli}, as well as \cite{Klar-conevolume-isotropic}, \cite{Klar-Mil-slicing-survey}. It is a natural question: \emph{is the Bobkov MGM position unique for a symmetric convex body?} We answer this question in the affirmative:
\begin{theorem}\label{uniqueness-Bobkov}
Let $K$ be a symmetric convex body. The expression $\sup_{T} \gamma(TK)$, where the supremum runs over all linear volume preserving operators $T$ on $\R^n$, is attained for the unique $T.$
\end{theorem}
\begin{proof} Without loss of generality, assume that $K$ is in the Bobkov MGM position. Suppose by contradiction that there exists a non-identity volume preserving linear map $T$ such that $\gamma(TK)=\gamma(K)$. Then there exists a traceless matrix $D$ such that $T=e^{D}$ (see e.g. \cite{Shiri-Eli} for the details), and by the rotation-invariance of the Gaussian measure we may assume that $D$ is diagonal. Let us now consider a function $F: [0,1]\rightarrow \R$ given by $F(t)=\gamma(e^{tD}K)$. On one hand, by the strong B-property of the Gaussian measure \cite{B-conj}, $F$ is log-concave on $[0,1].$ On the other hand, $F(0)=F(1)$ are maximal points for $F$, and therefore, the equality should be attained in the inequality $F(\frac{1}{2})\geq \sqrt{F(0)F(1)}.$ By Corollary \ref{strong-equal}, we get a contradiction with the assumption that $K$ is a convex \emph{body} (that is, a convex compact set with non-empty interior).
\end{proof}
We remark that the above proof is inspired by the works of Artstein-Avidan, Katzin \cite{Shiri-Katzin} and Artstein-Avidan, Putterman \cite{Shiri-Eli}, where in particular the authors consider \emph{maximal intersection position}: a symmetric convex body $K$ is said to be in the maximal intersection position if for any volume preserving linear operator $T$ on $\R^n,$ we have $\mu(K)\geq \mu(TK),$ where $\mu$ is the uniform measure on the centered euclidean ball of the same volume as $K.$ In both of the aforementioned papers, this position is viewed as part of different families of positions. It was conjectured in \cite{Shiri-Katzin}, and reiterated in \cite{Shiri-Eli}, that the maximal intersection position is indeed unique, and they explained (along the lines of the argument outlined above) that this conjecture would follow from the fact that the strong B-property for the uniform measure on the ball does not have non-trivial equality cases. We expect some of our ideas to be useful for studying this question, and leave it for future research.

In Section 2 we outline some preliminaries. In Section 3 we discuss some estimates concerning special functions related to the Gaussian measure. In Section 4 we outline the stability for the Gaussian Poincar{\'e} inequality restricted to a convex set, which is the result similar to what was obtained in \cite{Liv-ehr}, however we prove stability in a stronger distance. In Section 5 we outline the stability in the ``even version'' of the Gaussian Poincar{\'e} inequality restricted to a convex set, for quadratic forms -- this corresponds to stability in the ``local versions'' of the Theorems \ref{main-thm-global} and \ref{main-thm-strong-global}. Lastly, in Section 6 we prove Theorems \ref{main-thm-global} and \ref{main-thm-strong-global} and Corollary \ref{strong-equal}.

\textbf{Acknowledgements.} The second named author is supported by NSF DMS-1753260. The third named author is supported by ISF grant 1468/19 and BSF grant 2016050. The second and third authors are supported by NSF-BSF DMS-2247834. The fourth named author is supported by NSF DMS 2154402. The authors are grateful to ICERM for hospitality during the program ``Harmonic Analysis and Convexity''.

\section{Preliminaries}

Given a convex set $K$ in $\R^n$ and the standard Gaussian measure $\gamma$ on $\R^n,$ we shall use the notation
$$\frac{1}{\gamma(K)}\int_K d\gamma=\dashint_K d\gamma.$$

We will denote by $L^2(K,\gamma)$ the class of functions $u:\RR^n\to\RR$ such that $\int_K |u|^2 d\gamma<\infty$, with the normalized $L2$ norm $\|u\|^2_{L^2(K,\gamma)}=\dashint_K |u|^2 d\gamma$.


Given a convex set $K$ in $\R^n$, for a point $x\in\partial K$ we denote by $n_x$ the outward unit normal at $x$; the vector field $n_x$ is uniquely defined almost everywhere on $\partial K$. We say that $K$ is of class $C^2$ if its boundary is locally twice differentiable. In this case, $n_x$ is well-defined for all $x\in\partial K$. For a $C^2$ convex set, consider the second fundamental form of $K$ to be the matrix $\rm{II}=\frac{dn_x}{dx}$ (with a ``plus'' because the normal is outer) acting on the tangent space at $x$. The Gauss curvature at $x$ is $\det(\rm{II})$ and the mean curvature is $\tr(\rm{II}).$ 



We denote by $L$ the Ornstein-Uhlenbeck operator $L: C^2(\R^n)\rightarrow C(\R^n)$, given by
$$Lu=\Delta u-\langle \nabla u,x\rangle.$$The operator $L$ satisfies the following integration by parts identity whenever it makes sense (as follows immediately from the classical divergence theorem):
$$\int_{K} vLu \, d\gamma=-\int_{K}\langle \nabla v,\nabla u\rangle d\gamma+\int_{\partial K} v\langle \nabla u,n_x\rangle d\gamma_{\partial K}.$$
Here by $\gamma_{\partial K}$ we mean the measure on $\partial K$ with density $(2\pi)^{-n/2} e^{-\frac{|x|^2}{2}}$ with  respect to 
 $H_{n-1}$, the $(n-1)$-dimensional Hausdorff measure. 

We denote by $W^{k,2}(K,\gamma)$ the Sobolev space of all functions $u$ such that $u$ has weak partial derivatives up to order $k$, and all
those partial derivatives (including $u$ itself) belong to $L^2(K,\gamma)$. We will use the notation 
$$\|u\|^2_{W^{1,2}(K,\gamma)}=\dashint_K |\nabla u|^2 d\gamma $$
(even though strictly speaking this is not a norm on $W^{1,2}(K,\gamma)$, since $\|u\|^2_{W^{1,2}(K,\gamma)} = 0$ for constant functions).

Recall that the trace operator is a continuous linear operator 
$$\operatorname{TR}: W^{1,2}(K,\gamma)\rightarrow L^2(\partial K,\gamma)$$
such that for every $u\in C^1(K),$ continuous up to the boundary, we have 
$$\operatorname{TR}(u)=u|_{\partial K}.$$
We shall use informal notation $\int_{\partial K} u d\gamma$ to mean $\int_{\partial K} \operatorname{TR}(u) d\gamma$. 
Similarly, we use notation $\int_{\partial K} \langle \nabla u,n_x\rangle d\gamma$ to mean 
$\int_{\partial K} \langle \operatorname{TR}(\nabla u),n_x\rangle d\gamma$, where $\operatorname{TR}(\nabla u)$ is the vector formed by 
the trace functions of the weak first partial derivatives of $u$. We shall also use notation $\nabla$, $\Delta$ and so on to denote the appropriate quantities in the sense of weak derivatives. We will also use the following notation for the Gaussian perimeter:
$$\gamma^+(\partial K)=\int_{\partial K} d\gamma.$$

The following result appears in \cite{Liv-ehr}; we sketch its proof for completeness.

\begin{theorem}[Gaussian Trace Theorem for convex sets containing the origin]\label{GaussGarg}
Let $K$ be a convex domain such that $rB^n_2\subset K$ for some $r>0$. Fix $g\in W^{1,2}(K,\gamma)$. Then
$$\int_{\partial K} g^2 d\gamma_{\partial K}\leq \frac{1}{r}\int_{K} (ng^2+|\nabla g|^2) d\gamma.$$ 
\end{theorem}
\begin{proof} We use the estimate $\langle x,n_x\rangle\geq r$, and incorporate the trick similar to the ones from \cite{HKL}, \cite{KolMilsupernew} and use the divergence theorem.  We get
$$\int_{\partial K} g^2 d\gamma_{\partial K}\leq \frac{1}{r}\int_{\partial K} \langle g^2x,n_x\rangle d\gamma_{\partial K}=\frac{1}{r}\int_K (\div(g^2 x)-g^2|x|^2) d\gamma.$$
Note that
$$\div(g^2 x)=ng^2+2g\langle \nabla g,x\rangle\leq ng^2+|\nabla g|^2+g^2|x|^2.$$
Combining the above yields the result. \end{proof}

\section{Estimates on the special functions which concern the rate of the stability estimate}

In what follows, $C, c, C_1$ etc denote positive absolute constants that do not depend on the dimension and whose value may change from line to line.
Recall that the in-radius $r(K)$ of a convex set $K\subset\RR^{n}$
is the largest number $r>0$ such that $rB_{2}^{n}\subset K$. We
also denote by 
\[
\gamma^{+}\left(\partial K\right)=\int_{\partial K}\left(2\pi\right)^{-n/2}e^{-\left|x\right|^{2}/2}\dd H_{n-1}
\]
 the Gaussian surface area of $K$. 
 
 The main goal of this section
is to prove the following technical estimate: 
\begin{proposition}
\label{prop:technical-estimate}Let $K$ be a symmetric convex body
in $\RR^{n}$ with in-radius $r=r(K)>0$. Assume that for  $\delta<c_{0}$ we
have 
\[
\frac{\gamma(K)}{{\it \int_{rB_{2}^{n}}\left|x\right|^{2}\dd\gamma}}+\frac{\gamma(K)}{r\gamma^{+}(\partial K)}\ge\frac{1}{\delta}
\]
Then either $r\ge\sqrt{\log\frac{1}{\delta}}$ or $r\le C\sqrt{n}\delta^{\frac{1}{n+1}}$. 
\end{proposition}

For the proof we need several lemmas. First, we will need the Gaussian
isoperimetric inequality: Let $\Phi\left(x\right)=\gamma\left((-\infty,x]\right)$
denote the CDF of a standard normal random variable, and let $\Phi^{-1}:[0,1]\to\RR$
denote the inverse function. Define the isoperimetric profile $I:[0,1]\to\RR$
by $I(x)=\frac{1}{\sqrt{2\pi}}e^{-\frac{\Phi^{-1}(x)^{2}}{2}}$. Then
the Gaussian isoperimetric inequality (\cite{Bor-isop}, \cite{ST}) states that for every $K\subset\RR^{n}$
we have $\gamma^{+}\left(\partial K\right)\ge I\left(\gamma(K)\right)$.
Note that $I$ is concave and symmetric around $x=\frac{1}{2}$ where
it attains its maximum. 

We first give a lower bound on $\gamma^{+}\left(\partial K\right)$ in terms of $r$ instead of the measure $\gamma(K)$:
\begin{lemma}
\label{lem:iso-bounds}Let $K$ be a symmetric convex body in $\RR^{n}$
with in-radius $r>0$. 
\begin{enumerate}
\item \label{enu:big-measure}If $\gamma(K)\ge\frac{1}{2}$ then $\gamma^{+}\left(\partial K\right)\ge\frac{1}{\sqrt{2\pi}}e^{-r^{2}/2}.$ 
\item \label{enu:small-measure}If $\gamma(K)\le\frac{1}{2}$ then $\gamma^{+}\left(\partial K\right)\ge\left(\frac{cr}{\sqrt{n}}\right)^{n}e^{-r^{2}/2}.$ 
\end{enumerate}
\end{lemma}

\begin{proof}
$(1)$ By the definition of $r$ we know that $K$ is contained in a strip 
\[
S=\left\{ x:\ \left|\left\langle x,\theta\right\rangle \right|\le r\right\} 
\]
 for some $\theta\in S^{n-1}$. We therefore also have $K\subset H$
where $H=\left\{ x:\ \left\langle x,\theta\right\rangle \le r\right\} $,
and therefore $\frac{1}{2}\le\gamma(K)\le\gamma\left(H\right)=\Phi(r)$.
Since $I$ is decreasing on $\left[\frac{1}{2},1\right]$ it follows
from the Gaussian isoperimetric inequality that $\gamma^{+}\left(\partial K\right)\ge I\left(\Phi(r)\right)=\frac{1}{\sqrt{2\pi}}e^{-r^{2}/2}$. 

$(2)$ By the concavity of $I$ we have for all $0\le x\le\frac{1}{2}$ 
\[
I(x)=I\left(2x\cdot\frac{1}{2}+(1-2x)\cdot0\right)\ge2x\cdot I\left(\frac{1}{2}\right)=\sqrt{\frac{2}{\pi}}x.
\]
 Therefore, by the Gaussian isoperimetric inequality 
$$\gamma^{+}\left(\partial K\right)\ge I\left(\gamma(K)\right)\ge \sqrt{\frac{2}{\pi}}\cdot\gamma(K)\ge \sqrt{\frac{2}{\pi}}\cdot\gamma\left(rB_{2}^{n}\right).$$ 
Since the density of $\gamma$ on $rB_{2}^{n}$ is bounded from below
by $\frac{1}{\left(2\pi\right)^{n/2}}e^{-r^{2}/2}$ we have 
\[
\gamma\left(rB_{2}^{n}\right)\ge\frac{1}{\left(2\pi\right)^{n/2}}e^{-r^{2}/2}\cdot r^{n}\omega_{n},
\]
where $\omega_{n}$ denotes the volume of $B_{2}^{n}$. Since $\omega_{n}\ge\left(\frac{c}{\sqrt{n}}\right)^{n}$
we get 
\[
\gamma^{+}\left(\partial K\right)\ge\left(\frac{cr}{\sqrt{n}}\right)^{n}e^{-r^{2}/2}.
\]
 as claimed. 
\end{proof}
Next, we need some rough estimates for Gaussian integrals over balls. Sharper estimates are definitely known, 
but this lemma will suffice for our needs:
\begin{lemma}
\label{lem:ball-integrals}
\begin{enumerate}
\item We have $\gamma\left(2\sqrt{n}B_{2}^{n}\right)\ge\frac{3}{4}$ and
$\int_{2\sqrt{n}B_{2}^{n}}\left|x\right|^{2}\dd\gamma\ge cn$. 
\item For every $r>0$ we have $\int_{rB_{2}^{n}}\left|x\right|^{2}\dd\gamma\ge\left(\frac{c}{\sqrt{n}}\right)^{n}r^{n+2}e^{-r^{2}/2}$. 
\end{enumerate}
\end{lemma}

\begin{proof}
$(1)$ We use simple probabilistic bounds. Let $Z=\left(Z_{1},\ldots,Z_{n}\right)$
denote a standard normal random vector in $\RR^{n}$. Then 
\begin{align*}
\EE\left|Z\right|^{2} & =n\EE Z_{1}^{2}=n\\
\var\left|Z\right|^{2} & =n\var Z_{1}^{2}=2n.
\end{align*}
 It follows by Markov inequality that
\[
\PP\left(\left|Z\right|\ge2\sqrt{n}\right)=\PP\left(\left|Z\right|^{2}\ge4n\right)\le\frac{n}{4n}=\frac{1}{4},
\]
 or $\gamma\left(2\sqrt{n}B_{2}^{n}\right)\ge\frac{3}{4}$ as claimed.
By Chebyshev inequality we have for $n\ge15$ 
\[
\PP\left(\left|Z\right|<\frac{\sqrt{n}}{2}\right)=\PP\left(\left|Z\right|^{2}<\frac{n}{4}\right)\le\PP\left(\left|Z\right|^{2}<n-2\sqrt{2n}\right)\le\frac{1}{4},
\]
and therefore 
\[
\gamma\left(2\sqrt{n}B_{2}^{n}\setminus\frac{\sqrt{n}}{2}B_{2}^{n}\right)=\PP\left(\left|Z\right|\le2\sqrt{n}\right)-\PP\left(\left|Z\right|\le\frac{\sqrt{n}}{2}\right)\ge\frac{3}{4}-\frac{1}{4}=\frac{1}{2}.
\]
It follows that 
\[
\int_{2\sqrt{n}B_{2}^{n}}\left|x\right|^{2}\dd\gamma\ge\int_{2\sqrt{n}B_{2}^{n}\setminus\frac{\sqrt{n}}{2}B_{2}^{n}}\left|x\right|^{2}\dd\gamma\ge\frac{n}{4}\gamma\left(2\sqrt{n}B_{2}^{n}\setminus\frac{\sqrt{n}}{2}B_{2}^{n}\right)=\frac{n}{8},
\]
 and we choose $c>0$ so that the estimate will also hold for $1\le n\le14$. 
 
$(2)$ The density of $\gamma$ on $rB_{2}^{n}$ is bounded from below by
$\frac{1}{\left(2\pi\right)^{n/2}}e^{-r^{2}/2}$ and so 
\begin{align*}
\int_{rB_{2}^{n}}\left|x\right|^{2}\dd\gamma & \ge\frac{1}{\left(2\pi\right)^{n/2}}e^{-r^{2}/2}\cdot\int_{rB_{2}^{n}}\left|x\right|^{2}\dd x=\frac{1}{\left(2\pi\right)^{n/2}}e^{-r^{2}/2}\cdot n\omega_{n}\cdot\frac{r^{n+2}}{n+2}
\end{align*}
Again we use $\omega_{n}\ge\left(\frac{c}{\sqrt{n}}\right)^{n}$ to
get $\int_{rB_{2}^{n}}\left|x\right|^{2}\dd\gamma\ge\left(\frac{c}{\sqrt{n}}\right)^{n}r^{n+2}e^{-r^{2}/2}$. 
\end{proof}
Now we can prove Proposition \ref{prop:technical-estimate}
\begin{proof}[Proof of Proposition \ref{prop:technical-estimate}]
Under the assumption of the proposition we have either $\frac{\gamma(K)}{{\it \int_{rB_{2}^{n}}\left|x\right|^{2}\dd\gamma}}\ge\frac{1}{2\delta}$
or $\frac{\gamma(K)}{r\gamma^{+}(\partial K)}\ge\frac{1}{2\delta}$. 

Assume first that $\frac{\gamma(K)}{{\it \int_{rB_{2}^{n}}\left|x\right|^{2}\dd\gamma}}\ge\frac{1}{2\delta}$.
For $\delta$ small enough we must have $r\le2\sqrt{n}$: If this
is not the case then by Lemma \ref{lem:ball-integrals}(1) we have
\[
2\delta\ge2\delta\gamma(K)\ge\int_{rB_{2}^{n}}\left|x\right|^{2}\dd\gamma\ge\int_{2\sqrt{n}B_{2}^{n}}\left|x\right|^{2}\dd\gamma\ge cn\ge c,
\]
 which is clearly impossible for $\delta$ small enough. 

Next, as in Lemma \ref{lem:iso-bounds} we use the fact that $K$
is contained in a strip $S$ of width $2r$ and hence $\gamma(K)\le\gamma(S)\le Cr.$
Using Lemma \ref{lem:ball-integrals}(2) we obtain 
\[
\left(\frac{c}{\sqrt{n}}\right)^{n}r^{n+2}e^{-r^{2}/2}\le\int_{rB_{2}^{n}}\left|x\right|^{2}\dd\gamma\le2\delta\cdot\gamma(K) \le Cr\delta,
\]
 and since $r\le2\sqrt{n}$ we get 
\[
\left(\frac{c}{\sqrt{n}}\right)^{n}r^{n+1}e^{-2n}\le C\delta,
\]
 or $r\le C\sqrt{n}\delta^{\frac{1}{n+1}}$ as we claimed. 

Finally, assume that $\frac{\gamma(K)}{r\gamma^{+}(\partial K)}\ge\frac{1}{2\delta}$.
Using again the fact that $\gamma(K)\le Cr$ we obtain
\[
\gamma^{+}\left(\partial K\right)\le\frac{2\delta}{r}\gamma(K)=C\delta.
\]
If $\gamma\left(K\right)\ge\frac{1}{2}$ we see from Lemma \ref{lem:iso-bounds}(1)
that $e^{-r^{2}/2}\le C\delta\le\sqrt{\delta}$ (assuming we take
$\delta\le\frac{1}{C^{2}}$), and then $r\ge\sqrt{\log\frac{1}{\delta}}$.
If on the other hand $\gamma(K)\le\frac{1}{2}$ we see from Lemma
\ref{lem:iso-bounds}(2) that 
\[
\left(\frac{cr}{\sqrt{n}}\right)^{n}e^{-r^{2}/2}\le C\delta.
\]
 However, again in this case it follows from Lemma \ref{lem:ball-integrals}(1)
that  $r\le2\sqrt{n}$, so we have $\left(\frac{cr}{\sqrt{n}}\right)^{n}e^{-2n}\le C\delta$
or $r\le C\sqrt{n}\delta^{\frac{1}{n}}$. 
\end{proof}

\begin{remark} It is an interesting question -- what (symmetric) convex set in $\R^n$ has the smallest Gaussian perimeter if the largest ball centered at the origin which is contained in this set has radius $r$. 

In the non-symmetric case, it is natural to conjecture that the answer is the ball for smaller $r$ and the half-space for larger $r$. In fact, if $\gamma(r B_2^n) \ge \frac{1}{2}$ then indeed one can prove using the Gaussian isoperimetric inequality (\cite{Bor-isop}, \cite{ST}) that the half-space minimize the perimeter for a fixed $r$. 

In the symmetric case, it is natural to conjecture that the answer is the ball for smaller $r$ and the symmetric strip for larger $r.$ Once again, for large enough $r$ this follows from the result of Lata\l{}a and Oleszkiewicz \cite{sconj}, who showed that $r(K)\gamma^+(\partial K)$ is minimized for a symmetric convex set when $K$ is a symmetric strip of the same Gaussian measure as $K$. In other words,
$$\gamma^{+}(\partial K)\geq \frac{2 J^{-1}_0(\gamma(K)) e^{-\frac{J^{-1}_0(\gamma(K))^2}{2}}}{\sqrt{2\pi} r}$$
where $J_0(a)=2\Phi(a)-1$ denotes the Gaussian measure of a symmetric strip of width $2a$. 
Let $\tilde{R}_n$ be the radius of the ball whose Gaussian measure is $J_0(1)$, and suppose that $r\geq \tilde{R}_n.$ Then $J_0(1) \le \gamma(K)\le J_0(r)$, and noting that $J^{-1}_0(a) e^{-\frac{J^{-1}_0(a)^2}{2}}$ is decreasing for $a\geq J_0(1),$ we conclude that
$$\gamma^{+}(\partial K)\geq \frac{2 r e^{-\frac{r^2}{2}}}{\sqrt{2\pi} r}=\frac{2}{\sqrt{2\pi}} e^{-\frac{r^2}{2}},$$
and the inequality is sharp when $K$ is a strip. We could have used this estimate that improves Lemma \ref{lem:iso-bounds}, but this sharper result would only affect our outcome in terms of the value of the absolute constants which we are not tracking.
\end{remark}

\section{Stability in the Poincar{\'e} inequality on the convex set}

The Gaussian Poincar{\'e} inequality on a convex set (which follows e.g. from the Theorem of Brascamp and Lieb \cite{BrLi}) states that
\begin{equation}\label{poin} 
\gamma(K)\int_K f^2 d\gamma-\left(\int_K f d\gamma \right)^2\leq\gamma(K)\int_K |\nabla f|^2 d\gamma.
\end{equation}

The result below is similar to Theorem 1.5 in \cite{Liv-ehr}, although the $W^{1,2}$ norm is replaced there by $L^1$ norm. Using the $W^{1,2}$ norm is crucial for us, and thus for completeness, we outline the full proof here:

\begin{theorem}[the quantitative stability in the Gaussian Poincar{\'e}]\label{gauss-poin-stab}

Suppose that for a convex set $K$ containing $r B^n_2$, a function $f\in W^{1,2}(K)\cap C^1(K),$ and for $\epsilon>0,$ we have
$$\dashint_K f^2 d\gamma-\left(\dashint_K f d\gamma\right)^2\geq\dashint_K |\nabla f|^2 d\gamma-\epsilon.$$
Then there exists a vector $\theta\in\R^n$ (possibly zero), which depends only on $K$ and $f$, such that
\begin{itemize}
\item $\int_{\partial K} \langle \theta, n_x\rangle^2 d \gamma_{\partial K}\leq \frac{2(n+1)\gamma(K)\epsilon}{r};$
\item $\|f-\langle x,\theta\rangle-\dashint_{K} f d\gamma\|^2_{W^{1,2}(K,\gamma)}\leq 4\epsilon.$
\end{itemize}

\end{theorem}
\begin{proof} We first assume that the boundary of $K$ is of class $C^{\infty}$. Without loss of generality we may also assume that $\int_K f d\gamma=0$. Consider the function $u\in W^{2,2}(K)\cap C^2(K)$ such that $Lu=f$ and $\langle \nabla u, n_x\rangle=0$ on $\partial K$ (see e.g. \cite{Liv-ehr} for existence and regularity). By Bochner's formula (see e.g. \cite{KolMil}), we write 
\begin{equation}\label{eq-22-again}
\begin{aligned}
\int_K f^2 d\gamma =  &-\int_K \left(2\langle \nabla f,\nabla u\rangle+ |\nabla u|^2\right) d\gamma \\
&-\int_{K} \|\nabla^2 u\|^2 d \gamma-\int_{\partial K}  \langle \mbox{\rm{II}} \nabla_{\partial K} u, \nabla_{\partial K}  u\rangle   \,d\gamma_{\partial K} (x).
\end{aligned}
\end{equation}
Combining (\ref{eq-22-again}) with the Cauchy's inequality
\begin{equation}\label{cauchy-1}
\int_K \left(2\langle \nabla f,\nabla u\rangle+ |\nabla u|^2\right) d\gamma\geq -\int_K |\nabla f|^2 d\gamma,
\end{equation}
and an application of convexity of $K$
$$\int_{\partial K}  \langle \mbox{\rm{II}} \nabla_{\partial K} u, \nabla_{\partial K}  u\rangle )  \,d\gamma_{\partial K} \geq 0,$$
we get
\begin{equation}\label{eq-33}
\dashint_K f^2 d\gamma\leq
\dashint_K |\nabla f|^2 d\gamma-\dashint_{K} \|\nabla^2 u\|^2 d \gamma.
\end{equation}
Under the assumption of the theorem, this yields
\begin{equation}\label{ineq-hess-u}
\dashint_{K} \|\nabla^2 u\|^2 d \gamma\leq \epsilon.
\end{equation}
Therefore, there exists a vector $\theta\in\R^n$ such that
$$u=-\langle x,\theta\rangle+v$$
with $\dashint_K \nabla v d\gamma=0$ and $\dashint_{K} \|\nabla^2 v\|^2 d \gamma\leq \epsilon.$ By the Poincar{\'e} inequality (\ref{poin}), this implies that $\dashint_K |\nabla v|^2 d\gamma\leq \epsilon$. By our choice of $u,$ we have
$$\langle \nabla v,n_x\rangle=\langle \nabla u,n_x\rangle + \langle \theta,n_x\rangle= \langle \theta,n_x\rangle.$$

In order to show the first assertion, we apply the Trace Theorem \ref{GaussGarg} to all the partial derivatives of $v$ and sum it up:
$$\frac{1}{\gamma(K)}\int_{\partial K} |\langle \nabla v,n_x\rangle|^2 d\gamma_{\partial K}\leq \frac{1}{\gamma(K)}\int_{\partial K} |\nabla v|^2 d\gamma_{\partial K}\leq $$$$\frac{n+1}{r}\dashint_K (|\nabla v|^2+\|\nabla^2 v\|^2) d\gamma\leq \frac{2(n+1)}{r}\epsilon.$$
The desired inequality follows if we remember that $\langle \theta,n_x\rangle=\langle \nabla v,n_x\rangle$, which completes the proof when $K$ is smooth. 

In order to show the second assertion, note that we used (\ref{cauchy-1}) while proving the Poincar{\'e} inequality, and therefore, the assumption of the theorem gives
$$\dashint_K \left(2\langle \nabla f,\nabla u\rangle+ |\nabla u|^2\right) d\gamma-\epsilon\leq -\dashint_K |\nabla f|^2 d\gamma,$$
which amounts to
\begin{equation}\label{cauchy-impl}
	\dashint_K |\nabla f+\nabla u|^2 d\gamma \leq \epsilon.
\end{equation}
We write
$$\|f-\langle x,\theta\rangle\|_{W^1(K,\gamma)}=\dashint_K |\nabla f-\theta|^2 d\gamma\leq $$$$2\dashint_K |\nabla f+\nabla u|^2 d\gamma + 2\dashint_K |\nabla v|^2 d\gamma \leq 4\epsilon,$$
where we used the properties of $v$ together with (\ref{cauchy-impl}), and the conclusion follows.

Next, assume $K$ is a general compact convex body. Choose a sequence $\left\{ K_{i}\right\} $ of $C^{\infty}$-smooth
convex bodies such that $K_{i}\subset K$ and $K_{i}\rightarrow K$
in the Hausdorff distance. Clearly 
\[
\dashint_{K_{i}}f^{2}d\gamma-\left(\dashint_{K_{i}}fd\gamma\right)^{2}\geq\dashint_{K_{i}}|\nabla f|^{2}d\gamma-\epsilon_{i}
\]
 for $\epsilon_{i}\to\epsilon$, and $K_{i}\supseteq r_{i}B_{2}^{n}$
for $r_{i}\to r$. Let $\theta_{i}\in\RR^{n}$ be the vector constructed
in the proof for the body $K_{i}$. Since
\begin{align*}
\left|\theta_{i}\right|^{2} & \le2\left(\dashint_{K_{i}}\left|\nabla f-\theta_{i}\right|^{2}d\gamma+\dashint_{K_{i}}\left|\nabla f\right|^{2}d\gamma\right)\\
 & \le8\epsilon_{i}+2\dashint_{K_{i}}\left|\nabla f\right|^{2}d\gamma\to8\epsilon+2\dashint_{K}\left|\nabla f\right|^{2}d\gamma
\end{align*}
 it follows that $\left\{ \theta_{i}\right\} $ is bounded. Therefore
by passing to a subsequence we may assume without loss of generality
that $\theta_{i}\to\theta \in \R^n$. The conclusion 
\[
\ensuremath{\left\|f-\langle x,\theta\rangle-\dashint_{K}fd\gamma\right\|_{W^{1,2}(K,\gamma)}^{2}\leq4\epsilon.}
\]
now follows by continuity. For the second conclusion, we note
that for a fixed $\eta\in\RR^{n}$ the convergence 
\[
\int_{\partial K_{i}}\left\langle \eta,n_{K_{i},x}\right\rangle ^{2}\dd\gamma_{\partial K_{i}}\to\int_{\partial K}\left\langle \eta,n_{K,x}\right\rangle ^{2}\dd\gamma_{\partial K}
\]
 follows e.g. from Proposition A.3 of \cite{Livshyts2019}. Using this
fact it is now straightforward to deduce that 
\[
\int_{\partial K}\left\langle \theta,n_{K,x}\right\rangle ^{2}\dd\gamma_{\partial K}=\lim_{i\to\infty}\int_{\partial K_{i}}\left\langle \theta_{i},n_{K_{i},x}\right\rangle ^{2}\dd\gamma_{\partial K_{i}}\le\frac{2(n+1)\gamma(K)\epsilon}{r}
\]
as claimed. 

Finally, if $K$ is not compact we approximate it by the bodies $K_{m}=K\cap\left(mB_{2}^{n}\right)$ for $m=1,2,3,...$. This time the convergence 
\[
\int_{\partial K_{m}}\left\langle \eta,n_{K_{m},x}\right\rangle ^{2}\dd\gamma_{\partial K_{m}}\to\int_{\partial K}\left\langle \eta,n_{K,x}\right\rangle ^{2}\dd\gamma_{\partial K}
\]
 follows from the fact that on $\partial K_m\cap\partial K$ we
have $n_{K_{m},x}=n_{K,x}$ almost everywhere, and the contribution
of the integral on $\partial K_{m}\setminus\partial K\subset m\SS^{n-1}$
tends to zero as $m\to\infty$. The rest of the argument is
the same as before. 
\end{proof}

\section{Stability in the ``symmetric'' Gaussian Poincar{\'e} inequality for quadratic functions}

The main result of this section is stability (in some partial cases) of the ``symmetric'' Gaussian Poincar{\'e} inequality due to Cordero-Erasquin, Fradelizi and Maurey: if $f$ is an even function and $K$ is a symmetric convex set, then
$$\dashint_K f^2 d\gamma- \left(\dashint_K f d\gamma\right)^2\leq \frac{1}{2}\dashint_K |\nabla f|^2 d\gamma.$$

\begin{lemma}
\label{cor:poincare-bound}Let $K$ be a symmetric convex body such
that $K\supseteq rB_{2}^{n}$. Assume an odd function $f:\RR^{n}\to\RR$
satisfies 
\[
\dashint_K f^{2}\dd\gamma\ge\dashint_K\left|\nabla f\right|^{2}\dd\gamma-\epsilon
\]
 as well as 
\[
\nn{f-\left\langle x,\eta\right\rangle }^2_{L^{2}\left(\gamma_{K}\right)} < \epsilon
\]
 for some $\eta\in\RR^{n}$ and $\epsilon>0$. Then 
\[
\int_{\partial K}\left\langle \eta,n_{x}\right\rangle ^{2}\dd\gamma_{\partial K}\le C \left(\frac{\gamma^{+}(\partial K)}{\int_{rB_{2}^{n}}x^{2}\dd\gamma}+\frac{1}{r}\right)n\epsilon\gamma(K).
\]
\end{lemma}
\begin{proof}
By Theorem \ref{gauss-poin-stab} there exists $\theta\in\RR^{n}$ such that the conclusions (1) and (2) hold. Since $\dashint_K f\dd\gamma=0$ we have $\nn{f-\left\langle x,\theta\right\rangle }_{L^{2}\left(K, \gamma\right)}^{2}\le4\epsilon$.
Also, by our assumption we have $\nn{f-\left\langle x,\eta\right\rangle }_{L^{2}\left(K, \gamma\right)}^{2}\le\epsilon$.
Thus by the triangle inequality, we conclude
$$
\frac{1}{\gamma(K)}\int_{K}\left\langle x,\theta-\eta\right\rangle ^{2}\dd\gamma=\nn{\left\langle x,\theta\right\rangle -\left\langle x,\eta\right\rangle }_{L^{2}\left(K, \gamma\right)}^{2}\le$$ 
\begin{equation}\label{now-ref}
\nn{\left\langle x,\theta\right\rangle -f }_{L^{2}\left(K, \gamma\right)}^{2}+\nn{f -\left\langle x,\eta\right\rangle }_{L^{2}\left(K, \gamma\right)}^{2} \le (2\cdot 4+2)\epsilon=10\epsilon.
\end{equation}
 Using the fact that $K\supseteq rB_{2}^{n}$ we have from (\ref{now-ref}):
\[
\text{\ensuremath{\left|\theta-\eta\right|^{2}\cdot \frac{1}{n}\int_{r B^n_2} |x|^2 d\gamma=\int_{rB_{2}^{n}}\left\langle x,\theta-\eta\right\rangle ^{2}\dd\gamma\le\int_{K}\left\langle x,\theta-\eta\right\rangle ^{2}\dd\gamma\le10\gamma(K)\epsilon,} }
\]
 and thus 
 \begin{equation}\label{bnd-ref}
 \left|\theta-\eta\right|^{2}\le10\frac{n\gamma(K)}{\int_{r B^n_2} |x|^2 d\gamma}\epsilon.
 \end{equation}

It follows from (\ref{bnd-ref}) and the conclusion (2) of Theorem \ref{gauss-poin-stab}:
\begin{align*}
\int_{\partial K}\left\langle \eta,n_{x}\right\rangle ^{2}\dd\gamma_{\partial K} & \le2\left(\int_{\partial K}\left\langle \eta-\theta,n_{x}\right\rangle ^{2}\dd\gamma_{\partial K}+\int_{\partial K}\left\langle \theta,n_{x}\right\rangle ^{2}\dd\gamma_{\partial K}\right)\\
 & \le2\cdot\left(\frac{10n\gamma(K)}{\int_{r B^n_2} |x|^2 d\gamma}\epsilon\cdot\gamma^{+}\left(\partial K\right)+\frac{2(n+1)}{r}\gamma(K)\epsilon\right)\\
 & \le\left(\frac{20\gamma^{+}(\partial K)}{\int_{r B^n_2} |x|^2 d\gamma}+\frac{8}{r}\right)\gamma(K)n\epsilon
\end{align*}
\end{proof}

Lemma \ref{cor:poincare-bound} allows us to deduce:

\begin{theorem}\label{stab-poin-sym}
Let $K$ be a symmetric convex set with the in-radius $r$. Let $T$ be a positive definite matrix with columns $t_i=Te_i,$ $i=1,...,n,$ and let $s\geq 0$ denote the smallest eigenvalue of $T$. Assume that 
\[
\dashint_K \left\langle Tx,x\right\rangle^2 d\gamma - \left(\dashint_K \left\langle Tx,x\right\rangle d\gamma\right)^2
\ge2\dashint_K\left|Tx\right|^{2}\dd\gamma-\epsilon
\]
for small enough $\epsilon > 0$. Then for every $i=1,...,n,$ we have
$$\int_{\partial K}\left\langle t_{i},n_{x}\right\rangle ^{2}\dd\gamma_{\partial K}\le C \left(\frac{\gamma^{+}(K)}{\int_{r B^n_2} |x|^2 d\gamma}+\frac{1}{r}\right)n^2\epsilon\gamma(K).$$
Therefore, if $s>0$ and $\epsilon < c (\frac{s}{n})^2$ then either $r\geq \sqrt{\log \frac{c s^2}{n^2 \epsilon}},$ or $r\leq C\sqrt{n}\left(\frac{\epsilon}{s^2}\right)^{\frac{1}{n+1}}.$
\end{theorem}
\begin{proof}
Using the same approximation argument as in Theorem \ref{gauss-poin-stab}, we may assume $\partial K$ is $C^\infty$-smooth. Write $f=\left\langle Tx,x\right\rangle +c$ such that $\dashint_K f\dd\gamma=0$. Consider the function $u\in W^{2,2}(K)\cap C^2(K)$ such that $Lu=f$ and $\langle \nabla u, n_x\rangle=0$ on $\partial K$ (again see e.g. \cite{Liv-ehr} for existence and regularity). By Bochner's formula (again see e.g. \cite{KolMil}), we write 
\begin{align*}
\int_{K}f^{2}\dd\gamma & =-\int_{K}\left(\nn{\nabla^{2}u}^{2}+\left|\nabla u\right|^{2}+2\left\langle \nabla u,\nabla f\right\rangle \right)\dd\gamma-\int_{\partial K}\left\langle \II\nabla_{\partial K}u,\nabla_{\partial K}u\right\rangle \dd\gamma_{\partial K}.\\
 & \le-\int_{K}\left(\nn{\nabla^{2}u}^{2}+\left|\nabla u\right|^{2}+2\left\langle \nabla u,\nabla f\right\rangle \right)\dd\gamma,
\end{align*}
where in the last passage we used that $\rm{II}\geq 0$ since $K$ is convex. Since $\nabla u$ is odd we have $\int\nabla u\dd\gamma_{K}=0$, so
by the Poincar{\'e} inequality we have 
\begin{equation}\label{few1now}
\delta_{1}=\int_{K}\left(\nn{\nabla^{2}u}^{2}-\left|\nabla u\right|^{2}\right)\dd\gamma\ge0,
\end{equation}
 and thus
\begin{align}\label{ref2nowlll}
\int_{K}f^{2}\dd\gamma & \le-\delta_{1}-\int_{K}\left(2\left|\nabla u\right|^{2}+2\left\langle \nabla u,\nabla f\right\rangle \right)\dd\gamma\\
 & =-\delta_{1}-\int_{K}\left(\left|\sqrt{2}\nabla u+\frac{1}{\sqrt{2}}\nabla f\right|^{2}-\frac{1}{2}\left|\nabla f\right|^{2}\right)\dd\gamma.\nonumber
\end{align}
 It follows from (\ref{ref2nowlll}) that 
\begin{align*}
-\dashint_K \left\langle Tx,x\right\rangle^2 d\gamma + \left(\dashint_K \left\langle Tx,x\right\rangle d\gamma\right)^2
+2\dashint_K\left|Tx\right|^{2}\dd\gamma & =\frac{1}{\gamma(K)}\int_{K}\left(\frac{1}{2}\left|\nabla f\right|^{2}-f^{2}\right)\dd\gamma\\
 & \ge\frac{\delta_{1}}{\gamma(K)}+\dashint_K\left|\sqrt{2}\nabla u+\frac{1}{\sqrt{2}}\nabla f\right|^{2}\dd\gamma\ge0.
\end{align*}
 Therefore our assumption implies that 
\[
\dashint_K\left|\sqrt{2}\nabla u+\sqrt{2}Tx\right|^{2}\dd\gamma=\dashint_K\left|\sqrt{2}\nabla u+\frac{1}{\sqrt{2}}\nabla f\right|^{2}\dd\gamma\le\epsilon,
\]
 and so $\dashint_K\left|\nabla u+Tx\right|^{2}\dd\gamma\le\frac{\epsilon}{2}$.
In particular, for every $i$ we have
\begin{equation}\label{ref1now1}
\dashint_K\left|\partial_{i}u+\left\langle x,t_{i}\right\rangle \right|^{2}\dd\gamma\le\epsilon;
\end{equation}
recall that $t_{i}=Te_i\in\RR^{n}$ denotes the $i$'th row of $T$. 

However, our assumption also implies that 
\[
\dashint_K\left(\nn{\nabla^{2}u}^{2}-\left|\nabla u\right|^{2}\right)\dd\gamma=\frac{\delta_{1}}{\gamma(K)}\le\epsilon,
\]
 so in particular for all $i$ we have
 \begin{equation}\label{ref2now}
     \dashint_K\left(\partial_{i}u\right)^{2}\dd\gamma\ge\dashint_K\left|\nabla\partial_{i}u\right|^{2}\dd\gamma-\frac{\epsilon}{2}.
 \end{equation}
The first conclusion now follows from (\ref{ref1now1}) and (\ref{ref2now}), and Lemma \ref{cor:poincare-bound}:
\[
\int_{\partial K}\left\langle t_{i},n_{x}\right\rangle ^{2}\dd\gamma_{\partial K}\le C\left(\frac{\gamma^{+}(K)}{\int_{r B^n_2} |x|^2 d\gamma}+\frac{1}{r}\right)n\epsilon\gamma(K).
\]
 Summing over all $i$, and using the bound $|Tn_x|\geq s,$ we obtain the second conclusion:
\[
s^{2}\gamma^{+}\left(\partial K\right)\le\int_{\partial K}\left|Tn_{x}\right|^{2}\dd\gamma_{\partial K}\le C\left(\frac{\gamma^{+}(\partial K)}{\int_{r B^n_2} |x|^2 d\gamma}+\frac{1}{r}\right)n^2\epsilon\gamma(K),
\]
 or 
\[
\frac{\gamma(K)}{\int_{r B^n_2} |x|^2 d\gamma}+\frac{\gamma(K)}{r\gamma^{+}\left(\partial K\right)}\ge\frac{c s^{2}}{n^2 \epsilon}.
\]
Applying Proposition \ref{prop:technical-estimate} with $\delta=\frac{n^2 \epsilon}{c s^2},$ we get the second conclusion.
\end{proof}

From Theorem \ref{stab-poin-sym} we deduce

\begin{corollary}\label{stab-poin-sym-cor}
Let $K$ be a symmetric convex set with the in-radius $r$. Assume that 
\[
\dashint_K |x|^4 d\gamma - \left(\dashint_K |x|^2 d\gamma\right)^2
\ge 2\dashint_K |x|^2\dd\gamma-\epsilon
\]
for $\epsilon < \frac{c}{n^2}$.  Then either $r\geq \sqrt{\log \frac{c}{n^2 \epsilon}},$ or $r\leq C\sqrt{n}\epsilon^{\frac{1}{n+1}}.$
\end{corollary}

\section{Proofs of the main results.}

We first point out the following very nice fact:
\begin{lemma}\label{local-global} Suppose $V\in C^2(\R^n)$. Then 
\begin{equation}\label{assumption-equiv}
V\left(\frac{z_1+z_2}{2}\right)+\beta(z_1,z_2)= \frac{V(z_1)+V(z_2)}{2},
\end{equation}
where, letting $z(t)=\frac{(1-t)z_1+(1+t)z_2}{2},$ we have
\begin{equation}\label{beta}
\beta(z_1,z_2)=\frac{1}{8}\cdot\int_{-1}^1 (1-|t|)\langle \nabla^2 V(z(t))(z_1-z_2),z_1-z_2\rangle dt.
\end{equation}
\end{lemma}
\begin{proof} Note that for $b\in C^2[-1,1],$
$$\int_{-1}^1 (1-|t|)b''(t)dt=\int_0^1 (1-t)b''(t)dt+\int_{-1}^0 (1+t) b''(t)dt=$$$$(1-t)b'(t)|^1_{0}+\int_0^1 b'(t)dt+ (1+t)b'(t)|^{0}_{-1} -\int_{-1}^{0} b'(t)dt=$$
\begin{equation}\label{use1}
b(1)+b(-1)-2b(0).
\end{equation}
Let $b(t)=V(z(t))$. Then $b(-1)=V(z_1),$ $b(1)=V(z_2),$ $b(0)=V(\frac{z_1+z_2}{2}),$ and 
\begin{equation}\label{use2}
b''(t)=\frac{1}{4}\langle \nabla^2 V(z(t))(z_1-z_2),z_1-z_2\rangle.
\end{equation}
Combining (\ref{use1}) and (\ref{use2}) finishes the proof.
\end{proof}

\medskip

\textbf{Proof of the Theorem \ref{main-thm-global}.} Let $F(t)=\frac{\gamma(e^t K)}{\gamma(K)}.$ Then $F'(0)=\dashint (n-|x|^2)d\gamma,$ and
$$F''(0)=\dashint_K (n-|x|^2)^2 d\gamma-2\dashint_K |x|^2 d\gamma.$$
Next, let $V(t)=\log F(t)$. In view of the fact that $F(0)=1,$ we have $V''(0)=F''(0)-F'(0)^2$, and thus
$$V''(t)= \dashint_{e^t K} (n-|x|^2)^2 d\gamma-2\dashint_{e^t K} |x|^2 d\gamma -\left(\dashint_{e^t K} (n-|x|^2)d\gamma\right)^2.$$
On the other hand, by the assumption, letting $\alpha=\log a$ and $\beta=\log b,$ we have
\begin{equation}\label{V<eps}
V\left(\frac{\alpha+\beta}{2}\right)-\frac{1}{2}V(\alpha)-\frac{1}{2}V(\beta)\leq \log(1+\epsilon)\leq \epsilon.
\end{equation}
By Lemma \ref{local-global}, (\ref{V<eps}) implies that there exists a $t\in [\log a, \log b]$ such that
$$V''(t)=\dashint_{e^t K} |x|^4 d\gamma-2\dashint_{e^t K} |x|^2 d\gamma -\left(\dashint_{e^t K} |x|^2 d\gamma\right)^2\geq -\frac{8\epsilon}{(\log b/a)^2}.$$
An application of the Corollary \ref{stab-poin-sym-cor} with $e^tK$ in place of $K,$ and with $\frac{8\epsilon}{(\log b/a)^2}$ in place of $\epsilon$ finishes the proof. $\square$

\medskip

\textbf{Proof of the Theorem \ref{main-thm-strong-global}.} 
This time, we let $F_i(t)=\frac{\gamma(e^{te_i} K)}{\gamma(K)}.$ Then $F_i'(0)=\dashint (1-x_i^2)d\gamma,$ and
$$F_i''(0)=\dashint_K (1-x_i^2)^2 d\gamma-2\dashint_K x_i^2 d\gamma.$$
Next, for $x\in\R^n,$ let $V(x)=\log \frac{\gamma(e^x K)}{\gamma(K)}$. We have $\frac{\partial V}{\partial x_i\partial x_j}=0$ whenever $i\neq j,$ and for all $i=1,...,n,$
\begin{equation}\label{derivatives}
\frac{\partial^2 V}{\partial x^2_i}(0)=F_i''(0)-F_i'(0)^2.
\end{equation}
On the other hand, by our assumption,
\begin{equation}\label{v-est-eps}
V\left(\frac{x+y}{2}\right)-\frac{1}{2}V(x)-\frac{1}{2}V(y)\leq \log(1+\epsilon)\leq \epsilon.
\end{equation}
By Lemma \ref{local-global}, and in view of (\ref{derivatives}), (\ref{v-est-eps}) implies that there exists a $z\in [x,y]$ such that for all $i=1,..,n$
$$\dashint_{e^z K} x_i^4 d\gamma-2\dashint_{e^z K} x_i^2 d\gamma -\left(\dashint_{e^z K} x_i^2 d\gamma\right)^2\geq -\frac{8\epsilon}{(x_i-y_i)^2}.$$
By Theorem \ref{stab-poin-sym}, we get, denoting $\tilde{K}=e^z K$ and $\tilde{r}=|e^z| r(K)$, and summing up: 
\begin{equation}\label{keyconcl}
\int_{\partial \tilde{K}} \sum_{i=1}^n (x_i-y_i)^2 (n^i_x)^2 d\gamma_{\partial \tilde{K}}\leq C\epsilon n^2\gamma(\tilde{K})\left(\frac{1}{\tilde{r}}+\frac{\gamma^+(\partial \tilde{K})}{\int_{\tilde{r} B^n_2} |x|^2 d\gamma}\right).
\end{equation}
On the other hand, recalling the notation from the statement of the theorem, we have
\begin{equation}\label{keyconcl-lbnd}
\int_{\partial \tilde{K}} \sum_{i=1}^n (x_i-y_i)^2 (n^i_x)^2 d\gamma_{\partial \tilde{K}}\geq \delta^2 \int_{\partial \tilde{K}} \sum_{i\in \sigma^{\delta}} (n^i_x)^2 d\gamma_{\partial \tilde{K}}\geq \delta^2\alpha \gamma^+(\tilde{\Omega}_{\delta,\alpha}),
\end{equation}
where $\tilde{\Omega}_{\delta,\alpha}=e^z \Omega_{\delta,\alpha}.$ 

Combining (\ref{keyconcl}), (\ref{keyconcl-lbnd}) and Proposition \ref{prop:technical-estimate} (applied with $\tilde{K}$, $\tilde{r}$ and $\delta^2\epsilon \alpha\beta$), and recalling that $r(K)\in [|e^x|\tilde{r}, |e^y|\tilde{r}],$ we get the conclusion. $\square$

\medskip

Finally, we are moving to proving Corollary \ref{strong-equal}. First, we shall need the following

\begin{lemma}\label{lemma-measure-zero}
Let $H\subset\RR^{n}$ be a subspace. Let $K\subset\RR^{n}$ be
a closed, convex set with non-empty interior. Assume that at almost every point $x\in\partial K$
a unique normal $n_{x}$ exists and $n_{x}\in H$ (Here "almost everywhere" is with respect to the $(n-1)$-dimensional Hausdorff measure on $\partial K$). Then there exists a closed convex set $K_{0}\subset H$ such that $K=K_{0}\times H^{\perp}$. 
\end{lemma}
\begin{proof}
By translating $K$ we may assume without loss of generality that
$0$ is an interior point of $K$. Define $K_{0}$ to be the orthogonal
projection $K_{0}=\proj_{H}K$. Then we clearly have $K\subset K_{0}\times H^{\perp}$,
and $K_{0}$ also has $0$ in its (relative) interior. 

We first show that $\partial K\subset\partial K_{0}\times H^{\perp}$.
Indeed, fix $x\in\partial K$. If we define 
\[
\Omega=\left\{ y\in K:\ n_{y}\text{ is unique and }n_{y}\in H\right\} 
\]
then by assumption $K\setminus\Omega$ has measure $0$, so in particular $\Omega$ is dense in $\partial K$. 
Choose a sequence $x_{k}\in\Omega$ such that $x_{k}\to x$. By compactness we may pass
to a subsequence and assume without loss of generality that $n_{x_{k}}\xrightarrow{k\to\infty}v\in H$. 

We know that $x_{k}+n_{x_{k}}^{\perp}$ is a supporting hyperplane
for $K$ for all $k$. Therefore for every $z\in K$ we have 
\[
\left\langle z,n_{x_{k}}\right\rangle \le\left\langle x_{k},n_{x_{k}}\right\rangle .
\]
 Taking the limit as $k\to\infty$ we see that for every $z\in K$
we have $\left\langle z,v\right\rangle \le\left\langle x,v\right\rangle $.
In particular, every point $z_{0}\in K_{0}$ may we written as $z_{0}=\proj_{H}\left(z\right)$
for some $z\in K$, and since $v\in H$ we have 
\[
\left\langle z_{0},v\right\rangle =\left\langle z,v\right\rangle \le\left\langle x,v\right\rangle =\left\langle \proj_{H}x,v\right\rangle .
\]
For every $\epsilon>0$ the point $z_{\epsilon}=\proj_{H}x+\epsilon v$
satisfies 
\[
\left\langle z_{\epsilon},v\right\rangle =\left\langle \proj_{H}x,v\right\rangle +\epsilon\left|v\right|^{2}>\left\langle \proj_{H}x,v\right\rangle ,
\]
 so $z_{\epsilon}\notin K_{0}$. Therefore $\proj_{H}x\in\partial K_{0}$,
so indeed $x\in\partial K_{0}\times H^{\perp}$ as claimed. 

Now we prove that $K_{0}\times H^{\perp}\subset K$, finishing the
proof. Indeed, for every $x\in K_{0}\times H^{\perp}$ consider 
\[
A=\left\{ t\in[0,1]:\ tx\in K\right\} .
\]
 We know that $0\in A$ and $A$ is closed. Define $t_{0}=\max A$.
If $t_{0}<1$ it follows that $t_{0}x\in\partial K\subset\partial\left(K_{0}\times H^{\perp}\right)$,
which is impossible since $x\in K_{0}\times H^{\perp}$. Hence $t_{0}=1$
and $x\in K$ as claimed. 
\end{proof}

\medskip

\textbf{Proof of Corollary \ref{strong-equal}.} We may approximate $K$ with a convex set whose boundary is $C^2$ arbitrarily closely. We let $\sigma$ so that
$$\sigma^c=\{i\in [n]:\,x_i=y_i\},$$
$$H=\operatorname{span}\{e_i:\, i\in\sigma^c\},$$
and 
$$\Omega=\{x\in \partial K:\, n_x\not\in H^{\perp}\}.$$

We see that the assumption of Theorem \ref{main-thm-strong-global} is satisfied for an arbitrarily small $\epsilon>0.$ Pick $\alpha=\beta=\delta=\epsilon^{\frac{1}{10}}$ and let $\epsilon\rightarrow 0.$ Note that as $\epsilon\rightarrow 0,$ we might get different $z=z(\epsilon)$ in the conclusion of Theorem \ref{B-strong}, but as $z\in [x,y]$, then by compactness we may select a convergent sub-sequence of $z(\epsilon)$ to some point $z_0.$ We conclude that either $r(e^{z_0}K)=0$ (in which case $K$ has an empty interior), or $r(e^{z_0}K)=\infty$ (in which case $K=\R^n$), or $\gamma^+(\Omega)=0$. 
Since $\gamma_{\partial K}$ is absolutely continuous with respect to the Hausdorff measure on $\partial K,$ we conclude that in this case, for almost every $x\in\partial K,$ we have $n_x\in H^{\perp}.$ The conclusion follows from Lemma \ref{lemma-measure-zero}. $\square$

\bibliographystyle{plain}
\bibliography{Bstability}

\end{document}